\documentclass[reqno, oneside, 11pt]{amsart}
\usepackage[letterpaper]{geometry}
\usepackage{appendix}
\usepackage{amsmath}
\usepackage{amssymb}
\usepackage{enumerate,enumitem}
\usepackage{verbatim}
\usepackage{mathrsfs}
\usepackage{mathtools}
\usepackage{hyperref}
\usepackage{amsthm, mathdots, multicol,placeins}
\usepackage{xcolor}
\usepackage{ dsfont }

\DeclareMathOperator{\Ind}{Ind}
\DeclareMathOperator{\Jac}{Jac}
\DeclareMathOperator{\USp}{USp}
\DeclareMathOperator{\SU}{SU}
\DeclareMathOperator{\Sp}{Sp}
\DeclareMathOperator{\U}{U}
\DeclareMathOperator{\GL}{GL}
\DeclareMathOperator{\Id}{Id}
\DeclareMathOperator{\ST}{ST}
\DeclareMathOperator{\Gal}{Gal}
\DeclareMathOperator{\End}{End}
\DeclareMathOperator{\Aut}{Aut}
\DeclareMathOperator{\diag}{diag}
\DeclareMathOperator{\antidiag}{antidiag}
\DeclareMathOperator{\TL}{TL}
\DeclareMathOperator{\AST}{AST}
\DeclareMathOperator{\Hg}{Hg}
\DeclareMathOperator{\LL}{L}
\DeclareMathOperator{\tr}{tr}
\DeclareMathOperator{\Frob}{Frob}

\newtheorem{theorem}{Theorem}[section]
\newtheorem{example}[theorem]{Example}
\newtheorem{prop}[theorem]{Proposition}

\newtheorem{lemma}[theorem]{Lemma}
\newtheorem{defn}[theorem]{Definition}
\newtheorem{corollary}[theorem]{Corollary}
\newtheorem{conjec}[theorem]{Conjecture}

\theoremstyle{definition}

\theoremstyle{remark}
\newtheorem*{remark*}{Remark}

\author{Melissa Emory }
\address{Department of Mathematics, University of Toronto; 40 St. George Street, Toronto, Ontario, Canada M5S 2E4}
\email{memory@math.toronto.edu}

\author{Heidi Goodson}
\address{Department of Mathematics, Brooklyn College, City University of New York; 2900 Bedford Avenue, Brooklyn, NY 11210 USA}
\email{heidi.goodson@brooklyn.cuny.edu}

\title[Sato-Tate Distributions of $y^2=x^p-1$ and $y^2=x^{2p}-1$ ]{Sato-Tate Distributions of  $y^2=x^p-1$ and $y^2=x^{2p}-1$}

\begin{document}

%\tableofcontents

\begin{abstract}
  We determine the Sato-Tate groups and prove the generalized Sato-Tate conjecture for the Jacobians of curves of the form
$  y^2=x^p-1$ and $y^2=x^{2p}-1,$
where $p$ is an odd prime. Our results rely on the fact the Jacobians of these curves are nondegenerate, a fact that we prove in the paper. Furthermore, we compute moment statistics associated to the Sato-Tate groups. These moment statistics can be used to verify the equidistribution statement of the generalized Sato-Tate conjecture by comparing them to moment statistics obtained for the traces in the normalized $L$-polynomials of the curves.
\end{abstract}

\subjclass[2010]{11M50, 11G10, 11G20, 14G10}
\keywords{Sato-Tate groups, Sato-Tate distributions, Hyperelliptic curves}

\maketitle

\section{Introduction}

The original Sato-Tate conjecture is a statistical conjecture regarding the distribution of the normalized traces of Frobenius on an elliptic curve without complex multiplication (CM), and the conjecture was recently generalized to higher genus curves by Serre \cite{Ser12}.  Recent results on this topic of determining Sato-Tate distributions in genus 2 and 3 curves have been achieved in \cite{Arora2016, Fite2012, FKS2019, FiteLorenzo2018, FiteSuthlerland2014, Fite2016, KedlayaSutherland2009,  LarioSomoza2018}.  In 2016, Fit\'e-Gonz\'alez-Lario \cite{FGL2016} obtained Sato-Tate equidistribution results for a family of curves of arbitrarily high genus.  The main purposes of this paper are to compute the Sato-Tate groups and to  prove the generalized Sato-Tate conjecture for the following two families of hyperelliptic curves of arbitrarily high genus
$$C_p:y^2=x^p-1 \; \text{ and }\;C_{2p}:y^2=x^{2p}-1,$$
where $p$ is an odd prime.  The generalized Sato-Tate conjecture is known for CM abelian varieties due to the work of Johansson in \cite{Joh17}; in our proof for $C_p$ we follow Serre's strategy from \cite{Ser98}. We provide numerical evidence to support our results by computing moment statistics associated to the curves.

We start by recalling the original Sato-Tate conjecture for elliptic curves.  Let $F$ be a number field,  $E/F$ be an elliptic curve without complex multiplication, and $v$ be a finite prime of $F$ such that $E$ has good reduction at $v$.  By a theorem of Hasse,  the number of $\mathbb{F}_{q_v}$ points of $E$ is $q_v+1-a_v$, where $\mathbb{F}_{q_v}$ denotes the residue field of $v$ and  $a_v$ is an integer (called the trace of Frobenius at $v$) satisfying $|a_v| \leq 2{q_v}^{1/2}$.  The Sato-Tate conjecture predicts that,  as $v$ varies through the primes of good reduction for $E$, the normalized Frobenius traces $a_v/{q_v}^{1/2}$ are equidistributed in the interval $[-2,2]$ with respect to the image of the Haar measure on the special unitary group $\SU(2)$. This conjecture  has been proven for non-CM elliptic curves defined over totally real fields (see \cite{Barnet2011, Clozel2008,Harris2010,Taylor2008}).  The distributions of the normalized Frobenius traces are also known for CM elliptic curves over all fields: they are distributed with respect to the image under the trace map of the Haar measure on either the  unitary group $\U(1)$ or the normalizer of $\U(1)$ in $\SU(2)$, depending on whether or not the field of definition contains the field of complex multiplication (see, for example, \cite{allen2018potential} or \cite{Banaszak2015}).

The generalized Sato-Tate conjecture for an abelian variety predicts the existence of a compact Lie group that determines the limiting distribution of normalized local Euler factors.  We now state the conjecture more precisely for abelian varieties that are the Jacobians of curves defined over $\mathbb Q$, following the exposition of  \cite{FGL2016}.

Let $C$ be a smooth, projective, genus $g$ curve defined over $\mathbb Q$ and let $K$ be the minimal extension over which all endomorphisms of $\Jac(C)$ are defined. The Sato-Tate group of the Jacobian of $C$, $\ST(\Jac(C))\subseteq \USp(2g)$, is a compact Lie group satisfying the following property. For each prime $p$ at which $C$ has good reduction, there exists a conjugacy class $x_p$ of $\ST(\Jac(C))$ whose characteristic polynomial equals the normalized $L$-polynomial 
 \begin{equation}\label{eqn:normLpoly}\overline L_p(C,T) = T^{2g}+a_1T^{2g-1} +a_2 T^{2g-2} + \cdots + a_2T^2+a_1T+1.\end{equation}

Let $F$ be a number field and let $X_F$ be the set of conjugacy classes of $\ST(\Jac(C)_F)$. Let $\{p_i\}_{i \geq 1}$ be an ordering by norm of the set of primes of good reduction for $C$ over $F$ and define a map %$\Phi_i:\{p_i\}_{i\geq 1}\rightarrow X_{F}$
that sends $p_i$ to $x_{p_i}$ in $X_{F}$.  We can now state the generalized Sato-Tate conjecture.
\begin{conjec}\label{conjec:generalST}
(Generalized Sato-Tate Conjecture)
    The sequence $\{x_{p_i}\}_{i \geq 1}$ is equidistributed on $X_F$ with respect to the image on $X_F$ of the Haar measure of $\ST(\Jac(C)_F).$
\end{conjec}

Our goals in this paper are to determine the Sato-Tate groups of the Jacobians of the curves $C_p$ and $C_{2p}$ (see Theorem  \ref{theorem:STp} and Theorem \ref{thm:ST2p}) and to prove the equidistribution predicted by the generalized Sato-Tate conjecture for these two families of curves (see Theorem \ref{theorem:STconjec} and Theorem \ref{theorem:STconjecfor2p}).   
\begin{theorem}\label{thm:introST}
Let $p$ be an odd prime.  The generalized Sato-Tate conjecture holds for the Jacobians of the curves $C_p:\; y^2=x^p-1$ and $C_{2p}:\;y^2=x^{2p}-1$.
\end{theorem}

The Sato-Tate conjecture was proven for CM abelian varieties in \cite{Joh17}, and the Jacobians of both curves in Theorem \ref{thm:introST} are CM abelian varieties.  However,  to provide an explicit description of the limiting distribution of the normalized $L$-polynomial, we need the explicit embedding of the  Sato-Tate group of the Jacobian of the curve inside $\USp(2g)$, where $g$ is the genus of the curve. Our approach is similar to that of, for example, \cite{FGL2016} and \cite{LarioSomoza2018}. 

This paper is organized as follows.  In Section \ref{sec:nondegen}, we establish the nondegeneracy of the Jacobians of $C_p$ and $C_{2p}$. In  Section \ref{sec:algSTgroup}, we prove that the twisted Lefschetz group of the Jacobian, defined by Banaszak and Kedlaya in \cite{Banaszak2015}, is equal to the algebraic Sato-Tate group; this essentially follows from the work of \cite{Banaszak2003} since the Jacobians of our curves are nondegenerate. The equality of the twisted Lefschetz group and the algebraic Sato-Tate group allows one to interpret the Sato-Tate group as a maximal compact subgroup of the group of $\mathbb{C}$-points of the base change of the algebraic Sato-Tate group to $\mathbb{C}$. 

In Section \ref{sec:STgroups}, we apply the work of Section \ref{sec:algSTgroup} to determine the Sato-Tate groups of the Jacobians of the curves $C_p$ and $C_{2p}$. We first determine the identity components of the Sato-Tate groups (see Proposition  \ref{prop:idcomponent} and Proposition  \ref{prop:idcomponent2p}). Note that these propositions confirm Conjectures 6.8 and 6.9 of \cite{EmoryGoodsonPeyrot}. The computation of the twisted Lefschetz groups then gives the generators of the component groups (see Theorems \ref{theorem:STp} and \ref{thm:ST2p}). We give explicit examples of some of these generators in Table \ref{table:gammas} in Appendix \ref{app:component}.

In Section \ref{sec:equidist}, we establish Conjecture \ref{conjec:generalST} for $C_p$ and $C_{2p}$. For the curve $C_p$, we prove the generalized Sato-Tate conjecture by following Serre's strategy in \cite{Ser98}, i.e. showing that a certain $L$-function attached to the irreducible nontrivial representations of the Sato-Tate group of the Jacobian of the curve does not vanish.  We then  use a theorem of Hecke that the $L$-function attached to a nontrivial unitarized Hecke character does not vanish for Re$(s) \geq 1.$  This proof technique requires a cyclic Galois group $\Gal(K/\mathbb Q)$, where $K$ is the minimal extension over which all endomorphisms of the Jacobian are defined. The Galois group associated to the curve $C_{2p}$ is not cyclic, so we use the work of \cite{Joh17} to prove the generalized Sato-Tate conjecture in this case. 

In Section \ref{sec:moments}, we compute moment statistics associated to the Sato-Tate groups of the Jacobians of $C_p$ and $C_{2p}$. These moment statistics can be used to verify the equidistribution statement of the generalized Sato-Tate conjecture by comparing them to moment statistics obtained for the traces $a_i$ in the normalized $L$-polynomial $\overline L_p(C,T)$ in Equation \eqref{eqn:normLpoly}. Note that the numerical moment statistics are an approximation since one can only ever compute them up to some prime. It is of interest to those dealing with equidistribution statements to compare how close these two computations are.  We compare the moments in Table \ref{table:momemntsp} in  Section \ref{sec:moments}.

\subsection*{Notation and conventions} We begin by fixing notation used in later sections.  Let $C$ be a smooth projective curve defined over $\mathbb Q$. We write $\End(\Jac(C)_k)$ for the ring of endomorphisms defined over the field $k$ of the Jacobian of  $C$. Let $K:=K_C$ denote the minimal extension $L/\mathbb Q$ over which all the endomorphisms of the abelian variety $\Jac(C)$ are defined, i.e. the minimal extension for which $\End(\Jac(C)_L)\simeq \End(\Jac(C)_{\overline{\mathbb Q}})$; the field $K$ is called the endomorphism field of $\Jac(C)$.

We denote the Sato-Tate group of the Jacobian of $C$  by $\ST(\Jac(C)):=\ST(\Jac(C)_{\mathbb Q})$ with identity component denoted $\ST^0(\Jac(C)):=\ST^0(\Jac(C)_{\mathbb Q})$ and component group  $\ST(\Jac(C))/\ST^0(\Jac(C))$. The curve $y^2=x^m-1$ is denoted by $C_m$,  and when we specialize to $C_p$ or to $C_{2p}$ we assume throughout the paper that $p$ is an odd prime. We will write $\zeta_m$ for a primitive $m^{th}$ root of unity. For any rational number $x$ whose denominator is coprime to an integer $r$, $\langle x \rangle_r$ denotes the unique representative of $x$ modulo $r$ between 0 and $r-1$. 

Define the two matrices
$$I:=\begin{pmatrix}1&0\\0&1\end{pmatrix}\; \text{ and }\; J:=\begin{pmatrix}0&1\\-1&0\end{pmatrix}.$$ The symplectic form considered throughout the paper is given by $\diag(J,\dots,J).$
Lastly, for any positive integer $n$, we define the following subgroups of the unitary symplectic group $\USp(2n)$. 
\begin{equation*}\label{eqn:U1subn}
    \U(1)_n:=\left\langle \diag(\underbrace{u, \overline u, \ldots, u,\overline u}_{n-\text{times}}): u\in \mathbb C^\times, |u|=1\right\rangle
\end{equation*}
and 
\begin{equation*}\label{eqn:U1n}
    \U(1)^n:=\left\langle \diag( u_1,\overline{u_1},\ldots, u_n,\overline{u_n}):u_i\in \mathbb C^\times, |u_i|=1\right\rangle.
\end{equation*}

\subsection*{Acknowledgements} This material is based upon work supported by the National Security Agency under Grant No.\ H98230-19-1-0119, The Lyda Hill Foundation, The McGovern Foundation, and Microsoft Research, while the authors were in residence at the Mathematical Sciences Research Institute in Berkeley, California, during the summer of 2019. The first named author was supported by NSF grant DMS-2002085 and an AMS-Simons Travel Award.  The second author also received support for this project provided by a PSC-CUNY Award, jointly funded by The Professional Staff Congress and The City University of New York. 

We would like to thank Francesc Fit\'e for enlightening discussions while working on this project and Drew Sutherland for creating the histogram in Figure \ref{fig:x10histogram} for us (see also \cite{x10histogram}). We also thank Christelle Vincent, Holley Friedlander, and Fatma Cicek for their help with the Sage code to compute the matrices $\gamma$ and the characteristic polynomials during Sage Days 103. Lastly, we thank the reviewer for their thorough and  helpful comments.

%%%%%%%%%%%%%%%%%%%%%%%%%%%%%%%%%%%%%%%%%%%%%
%%%%%%%%%%%%%%%%%%%%%%%%%%%%%%%%%%%%%%%%%%%%%

\section{Nondegenerate Abelian Varieties}\label{sec:nondegen}

Our main results hold for curves  whose Jacobians are nondegenerate. In this section we define the term nondegenerate and give some known results that will be relevant to our later work.

Let $A$ be a nonsingular projective variety over $\mathbb C$. We denote (as in  \cite{Shioda82}) the (complexified) Hodge ring of $A$ by
$$\mathscr B^*(A):=\displaystyle\sum_{d=0}^{\dim(A)} \mathscr B^d(A), $$
where $\mathscr B^d(A)=(H^{2d}(A,\mathbb Q)\cap H^{d,d}(A))\otimes \mathbb C$ is the $\mathbb C$-span of Hodge cycles of codimension $d$ on $A$.  
Furthermore, we define the ring $$\mathscr D^*(A):=\displaystyle\sum_{d=0}^{\dim(A)} \mathscr D^d(A)$$ 
where $\mathscr D^d(A)$ is the $\mathbb C$-span of classes of intersection of $d$ divisors. This is the subring of $\mathscr B^*(A)$ generated by the divisor classes, i.e.  generated by $\mathscr B^1(A)$. In general, it is known that we have containment $\mathscr D^*(A) \subseteq \mathscr B^*(A)$ \cite{Shioda82}.

\begin{defn}\label{def:degenerate}\cite{Aoki2002}
An abelian variety $A$ is said to be \textbf{nondegenerate} if $\mathscr D^*(A) = \mathscr B^*(A)$. If $\mathscr D^*(A) \not= \mathscr B^*(A)$, then $A$ is said to be \textbf{degenerate}.
\end{defn}

\begin{defn}\label{def:stabdegenerate}\cite{Aoki2002,Hazama89}
An abelian variety $A$ is said to be \textbf{stably nondegenerate} if, for any integer $k\geq 1$, $\mathscr D^*(A^k) = \mathscr B^*(A^k)$. 
\end{defn}
Hazama proves in Theorem 1.2 of \cite{Hazama89} that $A$ is stably nondegenerate if and only if the dimension of its Hodge group is maximal. Note that in \cite{FGL2016}, the authors use the word \emph{nondegenerate} to describe abelian varieties with this property.

Nondegeneracy is related to the CM-type  of an abelian variety. Let $A/\mathbb Q$ be a dimension $d$ abelian variety. Suppose that there is a number field $K$ with $[K:\mathbb Q]=2d$ and an injective ring homomorphism $\iota: K \rightarrow \End(A)\otimes \mathbb Q$. The map $\iota$ induces a representation $\Phi$ of $K$ on the space of holomorphic 1-forms on $A$ and we say that $(A,\iota)$ is of \textbf{CM-type} $(F,\Phi)$ (see, for example, \cite{ShimuraTaniyama1961,Kubota1965,Aoki2002,Pohlmann1968}). When $A$ is an absolutely simple abelian variety with complex multiplication, 
stable nondegeneracy is equivalent to the CM-type being nondegenerate, i.e. the CM-type has maximal rank (see, for example, \cite{Kubota1965, Aoki2002}).

Let $\mathscr C^d(A)$ be the subspace of $\mathscr B^d(A)$ generated by the classes of algebraic cycles on $A$ of codimension $d$. Then 
$$\mathscr D^d(A) \subseteq \mathscr C^d(A) \subseteq \mathscr B^d(A)$$
and the Hodge Conjecture for $A$ asserts that $\mathscr C^d(A) = \mathscr B^d(A)$ for all $d$ \cite{Aoki2002, Shioda82}. It is clear from Definition \ref{def:degenerate} that if $A$ is nondegenerate then the Hodge Conjecture holds. However, there are many cases where the Hodge Conjecture holds for degenerate abelian varieties. For example,   Shioda verified the Hodge Conjecture for  $\Jac(C_m)$ for all $m\leq 21$, though $\Jac(C_9)$, $\Jac(C_{15})$, and $\Jac(C_{21})$ are degenerate (see  \cite[Section 6]{Shioda82}).

The following results are crucial to our work with the curves $C_p$ and $C_{2p}$.
\begin{prop}\label{prop:shiodadegenerate}\cite[Corollary 5.3]{Shioda82}
If $p\geq 3$ is a prime number, then the Hodge ring $\mathscr B^*(\Jac(C_p))$ is generated by  $\mathscr B^1(\Jac(C_p))$. The same result holds for arbitrary powers of $\Jac(C_p)$.
\end{prop}
By definition, this tells us that $\Jac(C_p)$ is stably nondegenerate. The nondegenerate CM-type for the curve $C_p$ is  $\{\mathbb{Q}(\zeta_p),\{\sigma_1,\sigma_2,\dots,\sigma_{(p-1)/2}\}\}$, where each $\sigma_t\in \Gal(\mathbb{Q}(\zeta_p)/\mathbb{Q})$ is defined by ${}^{\sigma_t}(\zeta_p)= \zeta_p^t$ (see, for example, \cite[Section 15.4]{ShimuraTaniyama1961}). The CM-type for the curve $C_p$ is primitive (see, for example, \cite{Kubota1965} for a definition) which implies that $\Jac(C_p)$ is absolutely irreducible.

\begin{prop}\label{prop:jacobianpowers}\cite[Lemma 4.1]{EmoryGoodsonPeyrot}
Let $g=2k$ be an even integer, and $C_{2g+2}:y^2= x^{2g+2}+c$, where  $c\in\mathbb Q^{\times}$. Then we have the following isogeny over $\overline{\mathbb Q}$
$$\Jac(C_{2g+2}) \sim \Jac({C_{g+1}})^2,$$
where $C_{g+1}: y^2= x^{g+1}+c$.
\end{prop}

Combining these two results yields the following.
\begin{corollary}\label{cor:stabdegenerate}
For any prime $p\geq 3$, $\Jac(C_p)$ and $\Jac(C_{2p})$ are nondegenerate.
\end{corollary}
\begin{proof}
Note that if $p$ is an odd prime, then we can write $p=2k+1$, for some integer $k$. Hence, $2p=2(2k+1)=2(2k)+2$ and Proposition \ref{prop:jacobianpowers} tells us that for $C_{2p}: y^2=x^{2p}-1$,
$$\Jac(C_{2p})\sim \Jac(C_p)^2.$$
Thus, $\Jac(C_{2p})$ is a power of $\Jac(C_{p})$, and we apply Proposition \ref{prop:shiodadegenerate} to get the desired result.
\end{proof}
An alternative proof that $\Jac(C_p)$ is nondegenerate is Theorem 2 in \cite{Kubota1965}. Note that that curve $C_{2p}$ has CM field $\mathbb Q(\zeta_{4p})$. See the proof of Theorem \ref{thm:ST2p} for the generators of the reduced automorphism group of $C_{2p}$.

%%%%%%%%%%%%%%%%%%%%%%%%%%%%%%%%%%%%%%%%%%%%%
%%%%%%%%%%%%%%%%%%%%%%%%%%%%%%%%%%%%%%%%%%%%%

\section{The Algebraic Sato-Tate Group}\label{sec:algSTgroup}
We start by defining notation as in \cite{FGL2016} and \cite[Section 3]{SutherlandNotes}. For more detailed background information, see \cite[Section 3.2]{SutherlandNotes}. 

Let $A/ k$ be an abelian variety of dimension $g$ defined over the number field $k$.  Let $\ell$ be a prime and we define the Tate module $T_{\ell}:=\varprojlim_{n} A[\ell^n]$ to be a free $\mathbb{Z}_{\ell}$-module of rank $2g$, and the rational Tate module $V_{\ell}:=T_{\ell}\otimes_{\mathbb{Z}} \mathbb{Q}$ to be a $\mathbb{Q}_{\ell}$-vector space of dimension $2g.$ The Galois action on the Tate module is given by an $\ell$-adic representation $$\rho_{A,\ell}:\Gal(\overline{{k}}/{k}) \rightarrow \Aut(V_{\ell}) \cong \GL_{2g}(\mathbb{Q}_{\ell}).$$
Let $G_{\ell}$ denote the image of this map, and let $G^{Zar}_{\ell}$ be the Zariski closure of $G_{\ell}$ in $\GL_{2g}(\mathbb{Q}_{\ell})$. We then define $G^{1,Zar}_{\ell}:=G^{Zar}_{\ell}\cap \Sp_{2g}(\mathbb{Q}_{\ell})$.
\begin{defn}
The Sato-Tate group of $A$, denoted by $\ST(A)$, is a maximal compact Lie subgroup of $G^{1,Zar}_{\ell}\otimes_{\mathbb{Q}_{\ell}}\mathbb{C}$ contained in $\USp(2g)$ 
\end{defn}

The algebraic Sato-Tate Conjecture for $\Jac(C)$ predicts the existence of an algebraic Sato-Tate group $\AST(\Jac(C))$ of $\Sp_{2g}/\mathbb{Q}$ such that 
$$G_{\ell}^{1,Zar}=\AST(\Jac(C)) \otimes_{\mathbb{Q}}\mathbb{Q}_{\ell}$$
for every prime $\ell$ (see, for example, \cite[Conjecture 2.13]{Fite2012} and \cite[Conjecture 2.1]{Banaszak2015}).

For each $\tau \in \Gal(\overline{\mathbb{Q}}/\mathbb{Q})$, define the set
$$\LL(\Jac(C))(\tau):=\{\gamma \in \Sp_{2g} | \gamma \alpha \gamma^{-1}=\tau(\alpha) \text { for all }\alpha \in \End(\Jac(C)_{\overline{\mathbb{Q}}})\otimes_{\mathbb{Z}} \mathbb{Q}\}$$ where $\alpha$ is viewed as an endomorphism of $H_1((\Jac(C)_{\mathbb{C}},\mathbb{Q}).$

\begin{defn}\label{def:twistedLef}\cite{Banaszak2015}
The twisted Lefschetz  group $\TL(\Jac(C))$ is defined to be 
\begin{equation*}\label{eqn:TL}
\TL(\Jac(C)):=\bigcup_{\tau \in \Gal(\overline{\mathbb{Q}}/\mathbb{Q})} \LL(\Jac C)(\tau).
\end{equation*}

\end{defn}

When $\tau$ is the identity automorphism, $\LL(\Jac(C))(\tau)$ forms a group, called the Lefschetz group, which we denote simply by $\LL(\Jac (C))$.

\begin{prop}\label{prop:AST=TL}
 Let $p$ be an odd prime and $C_p$ be the curve $y^2=x^p-1$. Then the algebraic Sato-Tate Conjecture holds for $\Jac(C_p)$ with $\AST(\Jac(C_p))=\TL(\Jac(C_p))$.
\end{prop}

\begin{proof}
This  follows from \cite[Theorem 6.6]{Banaszak2015} since $\Jac(C_p)$ is a nondegenerate CM abelian variety. Still, we include a proof that is similar to the proof of \cite[Lemma 3.5]{FGL2016} for the sake of completion. By \cite[Theorem 2.16(a)]{Fite2012}, we need to verify two criteria: the Hodge group $\Hg(\Jac(C_p))$ equals the Lefschetz group $\LL(\Jac(C_p))$, and that the Mumford-Tate Conjecture holds for $\Jac(C_p)$. The Mumford-Tate Conjecture is known to be true for CM abelian varieties (see, for example, \cite{FGL2016, Pohlmann1968, Yu2015}), so we only need to verify the first of the criteria. 

By Deligne  \cite[I, Proposition 6.2]{Del82} and \cite[Definition 4.4]{Banaszak2015}, we have
\begin{equation}\label{ineq:GZar0}
G^{1,Zar,0}_{\ell}(\Jac(C_p)) \subseteq \Hg(\Jac(C_p))\otimes_{\mathbb{Q}} \mathbb{Q}_\ell \subseteq \LL(\Jac(C_p)) \otimes _{\mathbb{Q}} \mathbb{Q}_\ell   
\end{equation}  
for every prime $\ell$. We will show that $G^{1,Zar,0}_{\ell}(\Jac(C_p))=\LL(\Jac(C_p)) \otimes _{\mathbb{Q}} \mathbb{Q}_\ell$ to obtain the desired result. Note that it is sufficient to show this for any prime $\ell$.

Since $p$ is prime, $\Jac(C_p)$ is simple (see, for example, \cite[Section 15.4]{ShimuraTaniyama1961}). Furthermore, Proposition \ref{prop:shiodadegenerate} tells us that $\Jac(C_p)$ has nondegenerate CM-type.  We apply the results of Section 2 of \cite{Banaszak2003} to get, for every prime $\ell$ of good reduction for which $\Jac(C_p)$  splits completely in $\mathbb Q(\zeta_p)$, 
\begin{equation}\label{equation:GZar0}
G_{\ell}^{1,Zar,0}(\Jac(C_p))=\{\diag( x_1,y_1,\dots,x_{g},y_{g}) \in \mathbb{Q}_{\ell}^{\times}~|~x_1y_1=\cdots=x_{g}y_{g}=1\},
\end{equation}
where $g=(p-1)/2$ is the genus of $C_p$. 

We now compute the Lefschetz group $\LL(\Jac(C))\otimes _{\mathbb{Q}} \mathbb{Q}_{\ell}$. In order for a matrix $\gamma\in\Sp_{2g}$ to commute with any matrix $\alpha \in \End(H_1(\Jac(C_p)_{\mathbb{C}},\mathbb{C}))$, it must be diagonal. Hence,
$$\LL(\Jac(C_p)) \otimes_{\mathbb{Q}} \mathbb{Q}_{\ell}=\{\diag( x_1,y_1,\dots,x_{g},y_{g}) \in \mathbb{Q}_{\ell}^{\times}~|~x_1y_1=\cdots=x_{g}y_{g}=1\},$$
which yields the desired result.
\end{proof}

\begin{corollary}\label{cor:AST=TL2p}
 If $p$ is an odd prime then the algebraic Sato-Tate Conjecture holds for $\Jac(C_{2p})$ with $\AST(\Jac(C_{2p}))=\TL(\Jac(C_{2p}))$.
\end{corollary}

\begin{proof}
Recall from Proposition \ref{prop:jacobianpowers} that 
$$\Jac(C_{2p})\sim (\Jac(C_{p}))^2.$$
Corollary \ref{cor:stabdegenerate} tells us that both $\Jac(C_{2p})$ and $\Jac(C_{p})$ are nondegenerate. Furthermore, they are both abelian varieties with CM. Hence, as in the proof of Lemma 3.5 of \cite{FGL2016}, proving the inclusions in Equation \eqref{ineq:GZar0} were actually equalities gives us
\begin{equation*}
G^{1,Zar,0}_{\ell}(\Jac(C_{2p})) = \Hg(\Jac(C_{2p}))\otimes_{\mathbb{Q}} \mathbb{Q}_\ell = \LL(\Jac(C_{2p})) \otimes _{\mathbb{Q}} \mathbb{Q}_\ell.   
\end{equation*}
\end{proof}

Note that we cannot apply Theorem A of \cite{Banaszak2003} to determine $G^{1,Zar,0}_{\ell}(\Jac(C_{2p}))$ since $\Jac(C_{2p})$ is not simple. We will determine the identity component of the Sato-Tate group of $\Jac(C_{2p})$ using another method in Section \ref{sec:2p}.

\begin{corollary}\label{cor:STgroupofcomponents} The group of components of $G_{\ell}^{1,Zar}(\Jac(C_p))$ and $\AST(\Jac(C_p))$ are isomorphic to $\Gal(\mathbb Q(\zeta_p)/\mathbb{Q})$. Also, the group of components of $G_{\ell}^{1,Zar}(\Jac(C_{2p}))$ and $\AST(\Jac(C_{2p}))$ are isomorphic to $\Gal(\mathbb Q(\zeta_{4p})/\mathbb{Q})$.
\end{corollary}
\begin{proof} This follows from Proposition \ref{prop:AST=TL} 
 (see  \cite[Prop 2.17]{Fite2012}).
\end{proof}
\begin{remark*}
Although \cite[Prop 2.17]{Fite2012} is stated for $g \leq 3$, Proposition \ref{prop:AST=TL} is for curves of arbitrarily high genus, and since the Mumford-Tate conjecture holds for $\Jac(C_p)$ the requirement that $g\leq 3$ in \cite[Prop 2.17]{Fite2012} can be removed for the proof of Corollary \ref{cor:STgroupofcomponents}.
\end{remark*}
It is known that when the algebraic Sato-Tate conjecture holds, we may interpret the Sato-Tate group $\ST(\Jac(C))$ as a maximal compact subgroup of $\AST(\Jac(C))\otimes_{\mathbb{Q}}{\mathbb{C}}$ (see, for example, \cite[Section 2.2]{Fite2012}).

%%%%%%%%%%%%%%%%%%%%%%%%%%%%%%%%%%%%%%%%%%%%%
%%%%%%%%%%%%%%%%%%%%%%%%%%%%%%%%%%%%%%%%%%%%%

\section{Sato-Tate Groups}\label{sec:STgroups}

In this section we compute the Sato-Tate groups of the Jacobians of the curves $C_p:\; y^2=x^p-1$ and $C_{2p}:\; y^2=x^{2p}-1$. For both families of curves, we obtain the component group of the Sato-Tate group by computing the twisted Lefschetz groups (recall the results of Proposition \ref{prop:AST=TL} and Corollary \ref{cor:AST=TL2p}).

%In the generic case, a genus $g$ curve $C$ will have $\ST(\Jac(C))\simeq \USp(2g)$ (see, for example, \cite{Fite2012, FKS2019, Katz99}). By ``generic" we mean that the Jacobian of the curve has trivial endomorphism ring. However, both of the families of curves we are working with have Jacobians with non-trivial endomorphism rings, so we obtain  closed proper subgroups of $\USp(2g)$ for the Sato-Tate groups.

\subsection{The Sato-Tate Group of $y^2=x^p-1$}\label{sec:STgroupp}

We first determine the identity component of the Sato-Tate group.

\begin{prop}\label{prop:idcomponent} If $p$ is an odd prime then  
$$\ST^0(\Jac(C_p))\simeq \U(1)^{g}$$
where $g=(p-1)/2$ is the genus of $C_p$.
\end{prop}

\begin{proof}
Let $\ell$ be a prime, and take an embedding of $\mathbb{Q}_{\ell}$ into the complex numbers.  By definition, $\ST^0(\Jac(C))$ is a maximal compact subgroup of $\AST^0(\Jac(C))\otimes_{\mathbb{Q}}  \mathbb{C}.$ From Proposition \ref{prop:AST=TL} and Equation \eqref{equation:GZar0}, it follows that we can take the maximal compact subgroup $\U(1)^g$.
\end{proof}

\begin{remark*}
Proposition \ref{prop:idcomponent} could also be derived from \cite{FGL2016} where they consider the curve $\mathcal{C}_k:v^{\ell}=u(u+1)^{\ell-k-1}$.  If we let $k=p-2$ and $\ell=p$, then the curve $\mathcal{C}_{p-2}$ is isomorphic to $C_p$ over the field $\mathbb{Q}(4^{1/p},i)$. This immediately gives the identity component of the Sato-Tate group of the Jacobian of $C_p$ since the connected component only depends on the variety over $\overline{\mathbb Q}$. %where the curve $C_p$ is realized as a quotient of the Fermat curve (\cite{Shioda82}) (to see the isomorphism,  do a change of variables sending $x\mapsto 4^{1/p}v$ and $y \mapsto \sqrt{-1}(2u-1)$). 

\end{remark*}
The main result of the following theorem is determining the component group of the Sato-Tate group of $\Jac(C_p)$.  Explicit examples of the generator of the component group are given in Table \ref{table:gammas} in Appendix \ref{app:component}.

\begin{theorem}\label{theorem:STp}
Let $S=\{1, \ldots, g\}$ and let  $a$ be a generator of the cyclic group $(\mathbb Z/p\mathbb Z)^{\times}$. Up to conjugation in $\USp(2g)$, 
$$\ST(\Jac(C_p))= \left\langle \U(1)^g, \gamma \right\rangle,$$
where $\gamma$ is a $2g\times 2g$ matrix whose block entries are given by
\begin{equation}\label{eqn:gammamatrix}
\gamma_{i,j} = \begin{cases}
I & \text{if } j=\langle ai\rangle_p \text{ and } \langle ai\rangle_p \in S,\\
J & \text{if } j=p-\langle ai\rangle_p \text{ and } \langle ai\rangle_p \not\in S,\\
0 & \text{otherwise.}
\end{cases}\end{equation}
Furthermore, there is an isomorphism $$\ST(\Jac(C_p))\simeq \U(1)^g\rtimes (\mathbb Z/p\mathbb Z)^{\times}.$$
\end{theorem}

\begin{proof}
We compute the twisted Lefschetz group of $\Jac(C_p)$. Applying Proposition \ref{prop:AST=TL} then yields the desired result.

  We can identify the group $G=(\mathbb Z/p\mathbb Z)^{\times}$ with $\Gal(\mathbb Q(\zeta_p)/\mathbb Q)$ via the isomorphism that maps $t\in G$ to the Galois element $\sigma_t$,%$G\rightarrow \Gal(F/\mathbb Q)$ defined by $ t\mapsto \sigma_t$,
% \begin{align*}
%     G&\rightarrow \Gal(F/\mathbb Q)\\
%     t&\mapsto \sigma_t,
% \end{align*}
where ${}^{\sigma_t}(\zeta_p):=\zeta_p^t$. %Here we are using left exponential notation for the Galois action.

A basis for the space of regular $1$-forms of a genus $g$ hyperelliptic curve   is given by  $\{\omega_j=x^j dx/y \;:\; j=0,\cdots, g-1\}$ (see, for example, \cite[Section 3]{VanWamelen1998}). We consider the automorphism $\alpha:C_p\rightarrow C_p$ defined by $\alpha(x,y)=(\zeta_p x,y)$, and compute the pullbacks of the differentials to be
 $$\alpha^*(\omega_j)=\zeta_p^{j+1}\omega_j .$$ 

We now write the endomorphism $\alpha \in \End(\Jac(C_K))$ in terms of a symplectic basis of $H_1(\Jac(C_p)_\mathbb{C},\mathbb C)$ (with respect to the matrix $\diag(J)$) and get the diagonal matrix
%$$\alpha=\begin{pmatrix}X_1&&&\\&X_2&&\\&&\ddots&\\&&&X_g \end{pmatrix},$$
$\alpha=\diag(X_1, X_2,\ldots, X_g),$
where each $X_i$ is a block matrix defined by
%$$X_i:=\begin{pmatrix}\zeta_p^i&0\\0&\overline{\zeta_p}^i \end{pmatrix}.$$
$$X_i:=\diag\left(\zeta_p^i,\overline{{\zeta_p}^i}\right).$$

Let $\sigma_a$ be a generator for the cyclic Galois group $\Gal(\mathbb Q(\zeta_p)/\mathbb Q)\simeq (\mathbb Z/p\mathbb Z)^{\times}$. Since the action of the Galois element $\sigma_a$ is given by ${}^{\sigma_a}(\zeta_p)=\zeta_p^{a}$, we have
$${}^{\sigma_a}X_i=\diag\left(\zeta_p^{ai},\overline{{\zeta_p}^{ai}}\right).$$
Hence, letting $S=\{1, \ldots, g\}$, we can write 
$${}^{\sigma_a}X_i=\begin{cases}
X_{\langle ai\rangle_p} & \text{if } \langle ai\rangle_p\in S,\\
\overline{X_{p-\langle ai\rangle_p}} & \text{if } \langle ai\rangle_p\not\in S,
\end{cases}$$
where
$$\overline{X_m}:=\diag\left(\overline{{\zeta_p}^m},{\zeta_p}^m\right).$$
Note that $JX_m(-J)=\overline{X_m}$. This characterization allows to express each ${}^{\sigma_a}X_i$ in the form $X_j$ or $\overline{X_j}$, for some $1\leq j\leq g$.

%Define $\gamma$  to be the  matrix defined in Equation \eqref{eqn:gammamatrix}. 
We will now verify that $\gamma\alpha\gamma^{-1}={}^{\sigma_a}\alpha$, where  $\gamma$ is as defined in Equation \eqref{eqn:gammamatrix}. Note that there is only one nonzero block entry in each row and each column in the block matrix $\gamma$. Furthermore, one easily checks that the entries of the inverse of $\gamma$ are given by
$${\gamma^{-1}}_{j,i} = \begin{cases}
I & \text{if } j=\langle ai\rangle_{p} \text{ and } \langle ai\rangle_{p} \in S,\\
-J & \text{if } j=p-\langle ai\rangle_{p} \text{ and } \langle ai\rangle_{p} \not\in S,\\
0 & \text{otherwise.}
\end{cases}$$

Some basic linear algebra shows that the only nonzero blocks in the product $\gamma\alpha\gamma^{-1}$ will be the diagonal entries. We will now determine what those diagonal entries will be. Suppose that the only nonzero block in column $j$ of $\gamma$ is in row $i$. Based on the definitions of $\gamma$ and $\gamma^{-1}$, this nonzero entry will yield the following product in the $i$th diagonal entry of $\gamma\alpha\gamma^{-1}$
$$\gamma_{i,j}X_j{\gamma^{-1}}_{j,i}=\begin{cases}
X_j & \text{if } j=\langle ai\rangle_{p} \text{ and } \langle ai\rangle_{p} \in S,\\
\overline{X_j} & \text{if } j=p-\langle ai\rangle_{p} \text{ and } \langle ai\rangle_{p} \not\in S.\\
\end{cases}$$
Hence, $\gamma\alpha\gamma^{-1}={}^{\sigma_a}\alpha$, which confirms that $\gamma$ is an element of the twisted Lefschetz group. 

We now show that $\gamma^{p-1}\in\ST^0(\Jac(C_p))$, but $\gamma^d\not \in \ST^0(\Jac(C_p))$ for any proper divisor $d$ of $p-1$, which will prove that $\ST(\Jac(C_p))=\langle \U(1)^g,\gamma\rangle\simeq \U(1)^g \rtimes (\mathbb Z/p\mathbb Z)^{\times}.$

Since $\sigma_a$ generates the Galois group $\Gal(\mathbb Q(\zeta_p)/\mathbb Q)$, we have $|\sigma_a|=2g$ and, for $1\leq d\leq2g$,
$$
(\sigma_a)^{d}(\zeta_p)=\begin{cases}
\zeta_p & \text{if } d=2g,\\
\overline{\zeta_p} & \text{if } d=g,
\end{cases}
$$
and $(\sigma_a)^{d}(\zeta_p)\not=\zeta_p$ nor $\overline\zeta_p$ otherwise. Hence, the action of $\sigma_a$ on the block matrix $X_i$ satisfies
$$ (\sigma_a)^{d}(X_i)=\begin{cases}
X_i & \text{if } d=2g,\\
\overline{X_i} & \text{if } d=g,\\
X_j \text{ or } \overline{X_j} & \text{otherwise},
\end{cases}$$
for some $j \in \{1, \ldots, g\}$ not equal to $i$.

We have seen that $\gamma \alpha \gamma^{-1}={}^{\sigma_a}\alpha$, and so conjugating $\alpha$ by $\gamma$ permutes (and sometimes conjugates) the diagonal block entries of $\alpha$. Since $\gamma \alpha \gamma^{-1}$ is again a diagonal block matrix, conjugating this by $\gamma$ will again just permute (and sometimes conjugate) the diagonal block entries. Hence, $\gamma^d\alpha\gamma^{-d}$ is a diagonal block matrix for any $d$. In fact, we can write $\gamma^d\alpha\gamma^{-d} ={}^{(\sigma_a)^d}\alpha$.

Thus, $\gamma^d$ has a nonzero, off-diagonal block entry if and only if there is some $i$ for which ${}^{(\sigma_a)^d}X_i=X_j$ or $\overline{X_j}$ with $j\not=i, p-i$. This is possible if and only if $d\not=2g$ or $g$.

If $d=g$, then ${}^{(\sigma_a)^d}X_i=\overline{X_i}$ for all $i$. Hence, all of the diagonal block entries of  $\gamma^g$ must be $J$ or $-J$ since $JX_i(-J)=-JX_iJ=\overline{X_i}$. Thus, $\gamma^g\not\in\ST^0(\Jac(C_p))$. However, $J^2=(-J)^2=-I$, so $\gamma^{2g}=-\Id$, which is an element of $\ST^0(\Jac(C_p))$. Thus,
$\ST(\Jac(C_p))\simeq \U(1)^g \rtimes (\mathbb Z/p\mathbb Z)^{\times}.$

\end{proof}

\subsection{The Sato-Tate Group of $y^2=x^{2p}-1$}\label{sec:2p}
  We use the results of Section \ref{sec:nondegen} and Proposition \ref{prop:idcomponent} to determine the identity component of the Sato-Tate group of $C_{2p}$.
  
\begin{prop}\label{prop:idcomponent2p} If $p$ is an odd prime and $C_{2p}: y^2=x^{2p}-1$, then  
$$\ST^0(\Jac(C_{2p}))\simeq (\U(1)_2)^{g/2}$$
where $g=p-1$ is the genus of $C_{2p}$.
\end{prop}

\begin{proof}
Recall from Proposition \ref{prop:jacobianpowers} that 
$$\Jac(C_{2p})\sim (\Jac(C_{p}))^2.$$
The curve $C_p$ has genus $g'=(p-1)/2=g/2$, and Proposition \ref{prop:idcomponent} gives  the identity component for the Sato-Tate group of its Jacobian. It follows that the identity component of $\ST^0(\Jac(C_{2p}))$ is $\ST^0(\Jac(C_p))$ embedded into $\USp(2g)$, yielding
$\ST^0(\Jac(C_{2p})) \simeq (\U(1)^{g/2})_2\simeq (\U(1)_2)^{g/2}.$
\end{proof}

The main result of the following theorem is determining the component group of the Sato-Tate group of $\Jac(C_{2p})$. This is an interesting addition to the literature  as the Sato-Tate groups of these curves do not have cyclic component groups.

\begin{theorem}\label{thm:ST2p}
Let $p$ be an odd prime,  $g=p-1$, $S=\{1, \ldots, g\}$, and $a$ be a generator of the cyclic group $(\mathbb Z/2p\mathbb Z)^*$. Up to conjugation in $\USp(2g)$, the Sato-Tate group of $C_{2p}: y^2=x^{2p}-1$ is
$$\ST(\Jac(C_{2p}))= \left\langle (\U(1)_2)^{g/2}, \gamma, \gamma' \right\rangle,$$
where the $2\times2$ block entries of $\gamma$ %is a $2g\times 2g$ matrix whose block entries 
are given by
$$\gamma_{i,j} = \begin{cases}
I & \text{if } j=\langle ai\rangle_{2p} \text{ and } \langle ai\rangle_{2p} \in S,\\
J & \text{if } i<\lfloor\frac{p}{2}\rfloor,\, j=p-\langle ai\rangle_{2p}, \text{ and } \langle ai\rangle_{2p} \not\in S,\\
-J & \text{if } i>\lfloor\frac{p}{2}\rfloor,\, j=p-\langle ai\rangle_{2p}, \text{ and } \langle ai\rangle_{2p} \not\in S,\\
0 & \text{otherwise,}
\end{cases}$$
for $1\leq i,j\leq g$, and $\gamma'=\diag(I,-I,\ldots,I,-I).$ Furthermore, there is an isomorphism $$\ST(\Jac(C_{2p}))\simeq (\U(1)_2)^{g/2}\rtimes \Gal(\mathbb Q(\zeta_{4p})/\mathbb Q).$$
\end{theorem}

See Table \ref{table:gammas} in Appendix \ref{app:component} for explicit examples of the matrix $\gamma$.

\begin{proof}

The reduced automorphism group of $C_{2p}$ is isomorphic to the dihedral group $D_{2p}$ (see, for example, \cite{Muller2017}). We consider the following generators of the automorphism group of  $C_{2p}$. Let  
$$\alpha(x,y)=({\zeta_{2p}} x,y)\; \text{ and }\; \beta(x,y)=(x^{-1}, {i}yx^{-p}),$$
where ${\zeta_{2p}}$ is a primitive $2p^{th}$ root of unity.   Thus, $\End(\Jac(C_{2p})_K)\simeq\End(\Jac(C_{2p})_{\overline{\mathbb Q}})$, where $K=\mathbb Q({\zeta_{2p},i})=\mathbb Q({\zeta_{4p}})$. %One can easily verify that $\beta$ is an involution.

We compute pullbacks of the differentials $\omega_j=x^jdx/y$, where $0\leq j<g=p-1$, in order to determine the generators of the endomorphism ring $\End(\Jac(C_{2p})_K)$. As in the proof of Theorem \ref{theorem:STp}, the pullback $\alpha^*$ leads to the endomorphism $\alpha=\diag(X_1,X_2,\ldots,X_g)$. Computing the pullback $\beta^*$ on the differential $\omega_j$ yields
$$\beta^*\omega_j =\frac{x^{-j}d(x^{-1})}{iyx^{-p}} = i\omega_{p-2-j}.$$
Thus, the endomorphism $\beta \in \End(\Jac(C_{2p})_K)$ is the  antidiagonal matrix $\beta=\antidiag(\underbrace{Z,Z,\ldots,Z}_{g})$, where $Z=\diag\left(i,-{i}\right).$\\

%For any odd prime $p$, $\mathbb Z/4p\mathbb Z^\times\simeq\mathbb Z/4\mathbb Z^\times\times\mathbb Z/p\mathbb Z^\times$. The group isomorphism maps $r\in\mathbb Z/4p\mathbb Z^\times$ to $(r\pmod4,r\pmod p)$. Both of the factors in the decomposition are cyclic, the first factor always being generated by $-1\pmod{4}$. Thus,  $\mathbb Z/4p\mathbb Z^\times$ can be generated by two elements.  We can choose the elements $a$ and $b$ that map to the generators  $(1,m)$ and $(-1,-1)$ under the group isomorphism, where $m$ is a generator of  $\mathbb Z/p\mathbb Z^\times$. 
We choose two elements $\sigma_a, \sigma_b$ that generate the Galois group $\Gal(K/\mathbb Q)\simeq ( \mathbb Z/4p\mathbb Z)^\times$ and satisfy
\begin{center}
\begin{multicols}{2}
$\sigma_a\colon\begin{cases}
    \zeta_{2p}\mapsto  \zeta_{2p}^a\\
    i \mapsto i
\end{cases}$

$\sigma_b\colon\begin{cases}
    \zeta_{2p}\mapsto  \zeta_{2p}\\
    i \mapsto -i,
\end{cases}$
\end{multicols}
\end{center}
where $a$ is a generator of $(\mathbb Z/4p\mathbb Z)^\times$.

Let $\gamma$ and $\gamma'$ be defined as in the statement of the theorem. One can verify that $${}^{\sigma_a}\alpha = \gamma\alpha\gamma^{-1},\;\;{}^{\sigma_a}\beta = \gamma\beta\gamma^{-1},\;\; {}^{\sigma_b}\alpha = \gamma'\alpha\gamma'^{-1}, \text{ and }{}^{\sigma_b}\beta = \gamma'\beta\gamma'^{-1}$$
using a similar strategy to the one used in the proof Theorem \ref{theorem:STp}, so we omit the proof here. In this case, the matrix $\gamma$ contains both $J$ and $-J$ as entries so that it conjugates $\beta$ properly. 

Lastly, one can show as in the proof of Theorem \ref{theorem:STp} that the component group of $\ST(\Jac(C_{2p}))$ is $\langle \gamma, \gamma'\rangle$.

\begin{comment}

Let $\gamma$ be defined as in the statement of the theorem. One can verify that $${}^{\sigma_a}\alpha = \gamma\alpha\gamma^{-1} \text{ and }{}^{\sigma_a}\beta = \gamma\beta\gamma^{-1}.$$
using a similar strategy to the one used in the proof Theorem \ref{theorem:STp}, so we omit the proof here. In this case, the matrix $\gamma$ contains both $J$ and $-J$ as entries so that it conjugates $\beta$ properly. 

Similarly, we can easily verify that %verify that $\gamma'$ (as defined in the statement of the theorem) is in the twisted Lefschetz group, we  note that conjugating the matrices $Z_4$ and $X_j$ by $\gamma'$ yields $\overline Z_4$ and $\overline X_j$, respectively. Thus, it is clear that
$${}^{\sigma_b}\alpha = \gamma'\alpha\gamma'^{-1} \text{ and }{}^{\sigma_b}\beta = \gamma'\beta\gamma'^{-1}.$$
Lastly, one can check that $\langle \gamma, \gamma'\rangle \simeq \Gal(\mathbb Q({\zeta_{4p}})/\mathbb Q)$. Thus, by Corollary \ref{cor:STgroupofcomponents}, the component group of $\ST(\Jac(C_{2p}))$ is $\langle \gamma, \gamma'\rangle$. 

\end{comment}
\end{proof}

\begin{corollary}
Up to conjugation in $\USp(2g,\mathbb C)$, the Sato-Tate group of $C_{2p}$ over $\mathbb Q(i)$ is
$$\ST(\Jac(C_{2p})_{\mathbb Q(i)})= \left\langle (\U(1)_2)^{g/2}, \gamma \right\rangle.$$
\end{corollary}
\begin{proof}
This follows from the fact that the minimal extension $L/\mathbb Q(i)$ over which all the endomorphisms of $\Jac(C)_{\mathbb Q(i)}$ are defined is $L=\mathbb Q({\zeta_{4p}})=\mathbb Q({\zeta_{2p,i}})$. 
\end{proof}

\section{Equidistribution Results}\label{sec:equidist} 

In this section we prove Theorem \ref{thm:introST}, which states that the generalized Sato-Tate conjecture holds for the Jacobians of $C_p$ and $C_{2p}$.  We first specify to the curve $C_p$. We begin by discussing  the $L$-functions associated to the  curve  and then state the generalized Sato-Tate conjecture. We then prove the generalized Sato-Tate conjecture following the strategy of Serre \cite{Ser98}. Finally, we prove the generalized Sato-Tate conjecture for the Jacobian of $C_{2p}$ using a result of \cite{Joh17}.

\subsection{Hecke Characters and $L$-Functions} We follow the exposition in \cite[Section 2.2]{FGL2016}, specifying to the curve $C_p$.
For a more thorough review of Hecke characters, we refer the reader to \cite{Lang90} and \cite{Weil52}. Let $\mathfrak{p}$ be a prime ideal to $p$ in $\mathbb{Q}(\zeta_p)$ and let $x$ be an element in the ring of integers of $\mathbb{Q}(\zeta_p)$.  Then there is precisely one $p^{th}$ root of unity $\chi_{\mathfrak{p}}(x)$ satisfying the condition
$$\chi_{\mathfrak{p}}(x)\equiv x^{(N(\mathfrak{p})-1)/p} \mod \mathfrak{p}.$$
We extend this to all of $\mathbb Q(\zeta_p)$ by setting $\chi_{\mathfrak{p}}(x)=0$ whenever $x\equiv 0 \pmod{\mathfrak p}$, and, thus, $\chi_{\mathfrak{p}}$ is a multiplicative character of order $p$ on $\mathbb F_{\mathfrak{p}}:=\mathcal O_{\mathbb Q(\zeta_p)}/\mathfrak p$.

We now define the Jacobi sums that appear in the $L$-functions of our curves. For all $h=(h_1,h_2) \in \mathbb{Z}/p\mathbb{Z} \times \mathbb{Z}/p\mathbb{Z}$, and for any ideal $\mathfrak p$ in $\mathbb Q(\zeta_p)$ not dividing $p$, we define  $$J_h(\mathfrak{p}):=-\sum\limits_{x\in \mathbb{F}_\mathfrak{p}} \chi_{\mathfrak{p}}(x)^{h_1}\chi_{\mathfrak{p}}(1-x)^{h_2}$$
(see \cite[Section 1.4]{Lang90}) and $J_h(\mathfrak{p})$ can be viewed as a function on $\mathbb{Z}/p\mathbb{Z} \times \mathbb{Z}/p\mathbb{Z}$ in terms of the characters on $\mathbb{Z}/p\mathbb{Z} \times \mathbb{Z}/p\mathbb{Z}$ (see \cite{Weil52}).  For each $h$ we extend the definition of $J_h(\mathfrak{p})$ to all ideals prime to $p$ in $\mathbb{Q}(\zeta_p)$ by multiplicativity.

\begin{lemma}
 Let $h$ be of the form $((p-2)a,a)$ for $a \in G=(\mathbb{Z}/p\mathbb{Z})^{\times}$ then 
 $$L((C_p)_{\mathbb{Q}(\zeta_p)},s)=L(J_h,s)^{p-1} \quad\text{ and }\quad L(C_p,s)=L(J_h,s).$$
\end{lemma}
\begin{proof}
This follows from the remark after Proposition 4.1.  One can also see this by computing the set $M_{p-2}$ as defined in \cite[(2.2)]{FGL2016}:
$$M_{p-2}=\{j \in G:\langle j \rangle_p<\langle (p-1)j \rangle_p\}=\{j\in G:\langle j \rangle_p < \langle -j \rangle_p\}=\{1,2,\dots,(p-1)/2\} $$ which gives the CM type for the curve $C_p$.  Using the Hecke characters $J_h$ for these $h$, \cite[Lemma 2.10] {FGL2016} gives the desired result. \end{proof}

\subsection{Generalized Sato-Tate Conjecture}  We specify the generalized Sato-Tate conjecture to the Jacobian of the curve $C_p$. Before we state the  conjecture, we need to set up some notation. 

Let $E/\mathbb Q$ be a subextension of $\mathbb Q(\zeta_p)/\mathbb Q$. Denote the set of conjugacy classes of $\ST(\Jac(C_p)_E)$ by $X_E$.  Let $P$ be an infinite subset of primes of a number field, and $\{\mathfrak{p}_i\}_{i\geq 1}$ be an ordering by norm of $P$.  Define a map $A_E:P \to X_E$ by sending $\mathfrak{p}$ to $x_{\mathfrak{p}}$. For any representation $\rho:\ST(\Jac(C_p)_E)\to GL_n(\mathbb{C})$ of $\ST(\Jac(C_p)_E)$, write
    $$L_{A_E}(\rho,s)=\displaystyle\prod_{\mathfrak{p}\in P} \det(1-\rho(x_{\mathfrak{p}})N(\mathfrak{p})^{-s})^{-1}.$$ 

We specify a theorem of Serre to the curve $C_p$ (see also \cite[Theorem 3.12]{FGL2016}). 
\begin{theorem}\cite[page I-23]{Ser98}\label{thm:serre} Suppose that for every irreducible nontrivial representation $\rho$ of $\ST(\Jac(C_p)_E)$ the Euler product $L_A(\rho,s)$ converges for Re$(s)>1$ and extends to a holomorphic and nonvanishing function for Re$(s)\geq1$. Then the sequence $\{x_{\mathfrak{p}_i}\}_{i\geq 1}$ is equidistributed over $X_E$ with respect to the projection on $X_E$ of the Haar measure of $\ST(\Jac(C_p)_E)$.
\end{theorem}

For a prime $q$ of $E$, let $x_q$ be the conjugacy class of $\ST(\Jac(C_p)_E)$ using the isomorphism $\ST(\Jac(C_p)_E)\simeq \ST(\Jac(C_p)_{\mathbb{Q}(\zeta_p)}) \rtimes \Gal(\mathbb{Q}(\zeta_p)/E)$.
Specifically, set
    $$x_q:=\left(\diag\left(\dfrac{J_{r_1}(\mathfrak{p})}{N(\mathfrak{p})^{1/2}},\dfrac{J_{r_1}(\mathfrak{\overline p})}{N(\mathfrak{p})^{1/2}},\dots, \dfrac{J_{r_{(p-1)/2}}(\mathfrak{p})}{N(\mathfrak{p})^{1/2}}, \dfrac{J_{r_{(p-1)/2}}(\mathfrak{\overline p})}{N(\mathfrak{p})^{1/2}}\right), \Frob_q\right) \in X_{\mathbb{Q}},$$
where each $r_i=((p-2)i, i)$. The set $\{r_1,r_2,\dots, r_{(p-1)/2}\}$ is a complete set of representatives of $M$, and $\mathfrak{p}$ is a prime of $\mathbb{Q}(\zeta_{p})$ lying over $q$. 

Now specify $P$ to be the set of primes of good reduction for $(C_p)_E$ and let $\{p_i\}_{i \geq 1}$ be an ordering by norm of $P$. We can now state the generalized Sato-Tate conjecture for $\Jac(C_p)$  (see, for example, \cite[page I-23]{Ser98}).     

\begin{conjec}[Generalized Sato-Tate]\label{conjec:genST}
The sequence $x_E:=\{x_{p_i}\}_{i \geq 1}$ is equidistributed on $X_E$ with respect to the image on $X_E$ of the Haar measure of $\ST(\Jac(C_p)_E).$
\end{conjec} 

The following theorem specifies this conjecture to $E=\mathbb Q(\zeta_p)$.

\begin{theorem}\label{thm:equidistonQzetap}
The generalized Sato-Tate conjecture holds for $\Jac(C_p)$ over $\mathbb{Q}(\zeta_p)$.
\end{theorem}
\begin{proof}
See \cite[Theorem 3.6]{Fite15}.
\end{proof}

To prove Conjecture \ref{conjec:genST} for $\Jac(C_p)$ over $\mathbb Q$, we will prove the convergence condition of Theorem \ref{thm:serre}. We first describe the irreducible representations of $\ST(\Jac(C_p))$ as in \cite{Ser77}. Let $\mathcal G=\ST(\Jac(C_p))$ so that $\mathcal G^0=\ST^0(\Jac(C_p))$. We associate to any tuple $\underline{b}=(b_1,b_2,\dots, b_{(p-1)/2}) \in \mathbb{Z}^{(p-1)/2}$ the irreducible representation  $\phi_{\underline{b}}:\U(1)^{(p-1)/2} \rightarrow \mathbb{C}^{\times}$ defined by 
$$\phi_{\underline{b}}(u_1,\ldots,u_{(p-1)/2})=\displaystyle\prod_{i=1}^{(p-1)/2}u_i^{b_i},$$
where $U=\diag(u_1,\overline u_1,\ldots,u_{(p-1)/2},\overline u_{(p-1)/2})\in \U(1)^{(p-1)/2}.$ 

Let $H_{\underline{b}} \subseteq \Gal(\mathbb{Q}(\zeta_p)/\mathbb{Q})$ be the subgroup such that
\begin{equation}\label{eqn:phibHaction}
    \phi_{\underline{b}}(u_1,\ldots,u_{(p-1)/2})=\phi_{\underline{b}}(^h(u_1,\ldots,u_{(p-1)/2}))
\end{equation}
for every $ h \in H_{\underline{b}}$.  Let $\mathcal H:=\mathcal G^0 \rtimes H_{\underline{b}}.$
Then we can extend $\phi_{\underline{b}}$ to $\mathcal H$ via the map 
$$\phi_{\underline{b}}:\mathcal H \to \mathbb{C}^{\times},\quad \phi_{\underline{b}}(u_1,\ldots,u_{(p-1)/2},h)=\displaystyle\prod_{i=1}^{(p-1)/2}u_i^{b_i}.$$  

By work of Serre \cite{Ser77}, every irreducible representation of $\mathcal G$ is of the form $\Theta:=\Ind_{\mathcal H}^{\mathcal G}(\chi \otimes \phi_{\underline{b}})$, where $\chi$ is a character of $H_{\underline{b}}$ viewed as a character of $\mathcal H$ using composition with the projection $\mathcal H \to H_{\underline{b}}$.

\begin{theorem}\label{theorem:STconjec}
The generalized Sato-Tate conjecture holds for $\Jac(C_p)$ over $\mathbb{Q}$.
\end{theorem}

\begin{proof}
We wish to apply Theorem \ref{thm:serre}, so we need to show 
$$ L_{A_{\mathbb{Q}}}(\Theta,s)=\displaystyle\prod_{p_i} \det(1-\Theta(x_{p_i})p_i^{-s})^{-1}$$ is holomorphic and non-vanishing on $Re(s) \geq 1$.

Let  $n$ be the cardinality of $H_{\underline{b}}$. We first consider the case where $\chi$ is the trivial character.  The theory of $L$-functions gives 
\begin{align*}
    L_{A_{\mathbb{Q}}}(\phi_{\underline{b}},s)&=L_{A_\mathbb Q}(\Ind_{\mathcal H}^{\mathcal G} \Ind_{\mathcal G^0}^{\mathcal H}\phi_{\underline{b}},s)\\
    &=L_{A_\mathbb Q}(n\Ind_{\mathcal H}^{\mathcal G} \phi_{\underline{b}},s)\\&=L_{A_\mathbb Q}(\Theta,s)^n.
\end{align*}
Note that the second equality holds by Equation \eqref{eqn:phibHaction}. By \cite[Section 3.5]{Fite15}, we then have $L_{A_{\mathbb{Q}}}(\phi_{\underline{b}},s)=L(\Psi,s)$  up to a finite number of Euler factors, where $\Psi$ is a Gr\"ossencharacter and $L(\Psi,s)$ is holomorphic and nonvanishing on Re$(s) \geq 1$. 

We now consider the case where $\chi$ is non-trivial.  Since $\Gal(\mathbb{Q}(\zeta_p)/\mathbb{Q})$ is cyclic there exists a character $\widetilde{\chi}$ of $\Gal(\mathbb{Q}(\zeta_p)/\mathbb{Q})$ such that $\widetilde{\chi}$ restricted to ${{H_{\underline{b}}}}$ equals $\chi.$ Thus,
$$\Theta=\Ind_{\mathcal H}^{\mathcal G} (\chi \otimes \phi_{\underline{b}})=\widetilde{\chi}\otimes \Ind_{\mathcal H}^{\mathcal G}\phi_{\underline{b}}.$$ Furthermore, 
$\Gal(\mathbb{Q}(\zeta_p)/\mathbb{Q})$ being cyclic also gives us that $$n\Theta=\widetilde{\chi}\otimes\Ind_{\mathcal G^0}^{\mathcal G} \phi_{\underline{b}}=\Ind_{\mathcal G^0}^{\mathcal G} \phi_{\underline{b}}.$$ 
Hence, we again have that $L_{A_\mathbb Q}(\Theta,s)^n=L(\Psi,n)$ up to a finite number of Euler factors, where $\Psi$ is a Gr\"ossencharacter and $L(\Psi,s)$ is holomorphic and nonvanishing on Re$(s) \geq 1$.  
\end{proof}
\begin{remark*}
The result also follows from \cite[Prop. 16]{Joh17}.
\end{remark*}

\begin{theorem}\label{theorem:STconjecfor2p}
Let $E/\mathbb Q$ be any subextension of $\mathbb Q(\zeta_{4p})/\mathbb Q$. Then the generalized Sato-Tate conjecture holds for $\Jac(C_{2p})$ over $E$.
\end{theorem}
\begin{proof} By Proposition \ref{prop:jacobianpowers},  $\Jac(C_{2p})\sim\Jac(C_p)^2$. The result then follows from \cite[Prop. 16]{Joh17}.
\end{proof}

%%%%%%%%%%%%%%%%%%%%%%%%%%%%%%%%%%%%%%%%%%%%%
%%%%%%%%%%%%%%%%%%%%%%%%%%%%%%%%%%%%%%%%%%%%%

\section{Moment Statistics}\label{sec:moments}

In this Section we compute moment statistics associated to the Sato-Tate groups. These moment statistics can be used to verify the equidistribution statement of the generalized Sato-Tate conjecture by comparing them to moment statistics obtained for the traces $a_i$ in the normalized $L$-polynomial $\overline L_p(C,T)$ in Equation \eqref{eqn:normLpoly}.

\subsection{Preliminaries}
The following background information has been adapted from \cite[Section 4]{LarioSomoza2018} and \cite[Section 4]{SutherlandNotes}. We start by recalling some basic properties of moment statistics. We define the $n$th moment (centered at 0) of a probability density function to be the expected value of the $n$th power of the values, i.e. $M_n[X]=E[X^n]$.

Recall  that for independent variables $X$ and $Y$ we have $E[X+Y]=E[X]+E[Y]$ and $E[XY]=E[X]E[Y]$ (see, for example, \cite{LarioSomoza2018}). Thus, we have the following
\begin{equation}\label{eqn:Mnproductbreakdown}
    M_n[XY]=M_n[X]M_n[Y],
\end{equation}
\begin{equation}\label{eqn:Mnpowersbreakdown}
    M_a[X]M_b[X]=M_{a+b}[X],
\end{equation}
and
\begin{equation}\label{eqn:Mnsumbreakdown}
    M_n[X_1+\cdots + X_m]=\sum_{a_1+\cdots+a_m=n} \binom{n}{a_1,\ldots, a_m}M_{a_1}[X_1]\cdots M_{a_m}[X_m].
\end{equation}
Furthermore, for any constant $b$, we have $M_n[b]=b^n$.

We will now work to define the Haar measure on the groups that we obtain for the identity component $\ST^0(\Jac(C))$. From Propositions \ref{prop:idcomponent} and \ref{prop:idcomponent2p} we see that the possible groups are 
$$\U(1)^g \text{ and } (\U(1)_2)^{g/2}.$$
For each of these groups, we are interested in the pushforward of the Haar measure onto the set of conjugacy classes conj($\U(1)^g$) or conj($(\U(1)_2)^{g/2}$). 

We start with the unitary group $\U(1)$ and consider the trace map $\tr$ on $U\in \U(1)$ defined by $z:=\tr(U)=u+\overline{u}=2\cos(\theta)$, where $u=e^{i\theta}$. This trace map takes values in $[-2,2]$. From here we see that $dz=2\sin(\theta)d\theta$ and 
$$\mu_{\U(1)}= \frac1\pi \frac{dz}{\sqrt{4-z^2}}=\frac1\pi d\theta$$
gives a uniform measure of $\U(1)$ on $\theta\in[-\pi,\pi]$ (see \cite[Section 2]{SutherlandNotes}). We can deduce the following pushforward measures
\begin{equation*}\label{eqn:muU1}
    \mu_{\U(1)^n}= \prod_{i=1}^n \frac1\pi \frac{dz_i}{\sqrt{4-z_i^2}}= \prod_{i=1}^n\frac{1}{\pi}d\theta_i\;\;
\text{ and }\;\;\mu_{(\U(1)_2)^{n}} = \prod_{i=1}^{n} \frac1\pi \frac{dz_i}{\sqrt{4-z_i^2}} = \prod_{i=1}^{n}\frac{1}{\pi}d\theta_i.
\end{equation*}

Note that though the measure $\mu_{(\U(1)_2)^{n}}$ is expressed the same as the measure $\mu_{\U(1)^{n}}$, we will get a different distribution since in the former case each eigenangle $\theta_i$ occurs with multiplicity 2 (see, for example, \cite[Section 4.3]{SutherlandNotes}).

We can now define the moment sequence $M[\mu]$, where $\mu$ is a positive measure on some interval $I=[-d,d]$. The $n^{th}$ moment $M_n[\mu]$  is, by definition, the expected value of $\phi_n$ with respect to $\mu$, where $\phi_n$ is the function $z\mapsto z^n$. It is therefore given by 
$$M_n[\mu] = \int_I z^n\mu(z).$$

For $\U(1)$ we have $M_n[\mu_{\U(1)}]  = \binom{n}{n/2}$, where $\binom{n}{n/2}=0$ if $n$ is odd. Hence,
$$M[\mu_{\U(1)}] = (1,0,2, 0,6,0,20,0,\ldots). $$
From here, we can compute $M_n[\mu_{\U(1)_2}] =2^n\binom{n}{n/2}$, and take binomial convolutions to obtain
\begin{align*}
M_n[\mu_{\U(1)\times \U(1)}]&=\sum_{r=0}^n\binom{n}{r}M_n[\mu_{\U(1)}]M_{n-r}[\mu_{\U(1)}].
\end{align*}
We can combine these strategies with Equations  \eqref{eqn:Mnproductbreakdown}, \eqref{eqn:Mnpowersbreakdown}, and \eqref{eqn:Mnsumbreakdown}  to compute moments for $\mu_{\U(1)^g}$ and $\mu_{(\U(1)_2)^{g/2}}$.

For each $i\in\{1,2, \ldots, g\}$, denote by $\mu_i$ the projection of the Haar measure onto the interval $\left[-\binom{2g}{i},\binom{2g}{i}\right].$ 
We can compute $M_n[\mu_i]$ by averaging over the components of the Sato-Tate group. For example, in the case where the curve has CM by $\mathbb Q(\zeta_d)$, we will denote the restriction of $\mu_i$ to the component $\ST^0(\Jac(C))\cdot\gamma^k$ by ${}^k\mu_i$ and 
 \begin{equation*}\label{eqn:momentsGaloisgrp}
     \mu_i=\frac1{d}\sum_{k=0}^{d} {}^k\mu_i \;\;\text{ and }\;\; M_n[\mu_i]= \frac1{d}\sum_{k=0}^{d} M_n[{}^k\mu_i].
 \end{equation*}

\subsection{Characteristic Polynomials}
In this subsection, we give results for the characteristic polynomials in each component of the Sato-Tate groups of $C_p$ and $C_{2p}$. 

\subsubsection{Characteristic Polynomials for $C_p$}\label{sec:charpolyp}
We start with a random matrix $U$ in the identity component $\ST^0(\Jac(C_p))$. We will denote  the characteristic polynomial of $U\gamma^i$ by   $P_{\gamma^i}(T)$. Since $\gamma^{p-1}\in \ST^0(C_{p})$, we only compute $P_{\gamma^i}(T)$ for $i=0,\ldots, p-2$.

\begin{example}\label{ex:charpoly11}
We compute the characteristic polynomials of the  curve $C_{11}\colon y^2=x^{11}-1$. This yields  $P_{\gamma^1}(T) = P_{\gamma^3}(T) = P_{\gamma^7}(T)=P_{\gamma^9}(T)=T^{10}+1$ and
\begin{align*}
P_{\gamma^0}(T)&=\prod_{i=1}^5(T-u_i)(T-\overline{u_i}),\\
P_{\gamma^2}(T)=P_{\gamma^6}(T)&=(T^{5}+{u_1}\overline{u_2}u_3u_4u_5)(T^5 +\overline{u_1}u_2\overline{u_3}\overline{u_4}\overline{u_5}),\\
P_{\gamma^4}(T)=P_{\gamma^8}(T)&=(T^{5}-{u_1}\overline{u_2}u_3u_4u_5)(T^5 -\overline{u_1}u_2\overline{u_3}\overline{u_4}\overline{u_5}),\\
P_{\gamma^5}(T)&=(T^2+1)^5.
\end{align*}
\end{example}

We have two general results for the characteristic polynomials associated to the Sato-Tate group of $C_p$, which we combine into the following proposition.

\begin{prop}\label{prop:pcharpolygamma0}
Let $C_p$ be the genus $g$ curve $y^2=x^{p}-1$, where $p=2g+1$ is prime. Then
$$P_{\gamma^0}(T)=\prod_{i=1}^g(T-u_i)(T-\overline{u_i})\;\;\;\;\text{ and }\;\;\;\;P_{\gamma^g}(T)=(T^2+1)^g.$$
\end{prop}
\begin{proof}

The first equality is a consequence of Proposition \ref{prop:idcomponent} which tells us that $\ST^0(\Jac(C_p))=\U(1)^g$.

For a justification of the second equality, we recall from our work in the proof of Theorem \ref{theorem:STp}. There we proved that $\gamma^g$ is a diagonal block matrix with $\pm J$ on its diagonal entries. Multiplying $U$ by $\gamma^g$ yields a  diagonal block matrix, whose diagonal blocks are of the form
$$\begin{pmatrix}0&u_i\\-\overline{u_i}&0 \end{pmatrix}\; \text{ or }\; \begin{pmatrix}0&-u_i\\\overline{u_i}&0 \end{pmatrix},$$
depending on whether we multiplied by $J$ or $-J$. In either case, the factor of the characteristic polynomial associated to this block is of the form
$$T^2+u_1\overline{u_1}=T^2+1.$$
Thus, since there are $g$ diagonal blocks, the characteristic polynomial is 
$$P_{\gamma^g}(T)=(T^2+1)^g.$$
\end{proof}

\subsubsection{Characteristic Polynomials for $C_{2p}$}

We again start with a random matrix $U$ in the identity component of the Sato-Tate group. Recall that the Sato-Tate group  of $C_{2p}$ has two generators for the component group: $\gamma$ and $\gamma'$.  We will denote  the characteristic polynomial of $U\gamma^i(\gamma')^j$ by   $P_{i,j}(T)$. Since $\gamma^{p-1},(\gamma')^2\in \ST^0(C_{2p})$, we only compute $P_{i,j}(T)$ for $i=0,\ldots, p-2$ and $j=0,1$. 

\begin{example}\label{ex:charpoly10}
We compute the characteristic polynomials of the  curve $C_{10}\colon y^2=x^{10}-1$. This yields $P_{1,0}(T)=P_{3,0}(T)=P_{1,1}(T)=P_{3,1}(T)=(T^4+1)^2$ and
\begin{align*}
P_{0,0}(T)&=\prod_{i=1}^2(T-u_i)^2(T-\overline{u_i})^2,\\
P_{2,0}(T)&=(T^2+1)^4,\\
P_{0,1}(T)&=T^8-2(u_1\overline u_2+\overline u_1u_2)T^6+(4+(u_1\overline u_2)^2+(\overline u_1u_2)^2)T^4\\
&\hspace{.2in}-2(u_1\overline u_2+\overline u_1u_2)T^2+1,\\
P_{2,1}(T)&=T^8+2(u_1\overline u_2+\overline u_1u_2)T^6+(4+(u_1\overline u_2)^2+(\overline u_1u_2)^2)T^4\\
&\hspace{.2in}+2(u_1\overline u_2+\overline u_1u_2)T^2+1.
\end{align*}

\end{example}

We have two general results for the characteristic polynomials associated to the Sato-Tate group of $C_{2p}$, which we combine into the following proposition.

\begin{prop}\label{prop:2pcharpolygamma0}
Let $C_{2p}$ be the genus $g$ curve $y^2=x^{2p}-1$, where $p$ is prime. Then
$$P_{0,0}(T)=\prod_{i=1}^{g/2}(T-u_i)^2(T-\overline{u_i})^2\;\;\;\;\text{ and }\;\;\;\;P_{g/2,0}(T)=(T^2+1)^g.$$
\end{prop}
\begin{proof}
The first equality is a consequence of Proposition \ref{prop:idcomponent2p} which tells us that $\ST^0(\Jac(C_p))=(\U(1)_2)^{g/2}$. For a justification of the second equality, see the proof of Proposition \ref{prop:pcharpolygamma0}.
\end{proof}

We also have the following conjecture.

\begin{conjec}\label{conjec:2pcharpolygammarelprime}

Let $C_{2p}$ be the genus $g$ curve $y^2=x^{2p}-1$, where $p$ is prime. Then
$P_{d,j}(T)=(T^{g} +1)^2,$ for any $d$ relatively prime to $2g$ and $j=0$ or 1. 
\end{conjec}

\subsection{General Results for the Moments}

Based on the results of Section \ref{sec:charpolyp}, we have the following general result for the moment statistics associated to the Sato-Tate group of $C_p$.
\begin{prop} 
For the curve $C_p$ we have
$$M_n[{}^g\mu_i]=\begin{cases}
\binom{g}{i/2}^n & \text{ if } i \text{ is even}\\
0 & \text{ otherwise}.
\end{cases}$$
\end{prop}

\begin{proof}
Recall from Proposition  \ref{prop:pcharpolygamma0}   that $P_{\gamma^g}(T)=(T^2+1)^g.$ Expanding this yields
$$P_{\gamma^g}(T)=\sum_{i=0}^g \binom{g}{j}T^{2j}.$$
Thus, $a_i = \binom{g}{i/2}$ when $i$ is even and it equals 0 when $i$ is odd. It is then clear that $\mu_i(\phi_n)$ in this case is $\binom{g}{i/2}^n$ when $i$ is even and 0 when $i$ is odd 
\end{proof}

We also have the following conjecture for  characteristic polynomials and moments.

\begin{conjec}\label{conjec:pmomentrelprime}
Let $C_p$ be the genus $g$ curve $y^2=x^{p}-1$, where $p=2g+1$ is prime. Then
$P_{\gamma^d}(T)=T^{2g} +1,$ for any $d$ relatively prime to $2g$ and
$M_n[{}^k\mu_i]=0$.
\end{conjec}

\subsection{Explicit Examples of Moment Statistics}\label{sec:x11moments}
We first determine moment statistics for the genus 5 curve $C_{11}:\; y^2=x^{11}-1$. Using characteristic polynomials $P_{\gamma^k}(T)$ that were computed for each component in Example \ref{ex:charpoly11} and the properties in Equations \eqref{eqn:Mnproductbreakdown}, \eqref{eqn:Mnpowersbreakdown}, and \eqref{eqn:Mnsumbreakdown}, we can compute the $n$th moments for each $\mu_i$, $1\leq i\leq5$.  These moments, given in Table \ref{table:momentsp}, are easily computed using Sage \cite{Sage}.  See Table \ref{table:momemntsp} in Section \ref{sec:a1mu1Moments} for a comparison of $M[\mu_1]$ to the numerical  moments $M[a_1]$ of the normalized $L$-polynomial of the curve.
\begin{table}[h]
{\renewcommand{\arraystretch}{1.5}
\begin{tabular}{|c|l|}
\hline
$M[\mu_1]$& $(1, 0, 1, 0, 27, 0, 1090, \ldots)$\\
\hline
$M[\mu_2]$&$(1, 1, 9, 133, 2873, 75453, 2200605, \ldots)$\\
\hline
$M[\mu_3]$&$(1, 0, 24, 0, 1381080, 0, 161935061760, \ldots)$\\
\hline
$M[\mu_4]$&$(1, 2, 64, 4688, 498236, 61887736, 8430343600,  \ldots)$\\
\hline
$M[\mu_5]$&$(1, 0, 72, 0, 934332, 0, 22782049800, \ldots)$\\
\hline
\end{tabular}}
\caption{Moment Statistics for $y^2=x^{11}-1$.}\label{table:momentsp}
\end{table}
%\end{prop}

Using the same strategy as above, we determine that the $\mu_1$-moment statistics for the Sato-Tate group of $C_{10}\colon y^2=x^{10}-1$ are 
$$M[\mu_1]=(1, 0, 2, 0, 72, 0, 3200,0,156800,0,8128512 \ldots).$$
In Figure \ref{fig:x10histogram} we give a histogram of $a_1$-values of $y^2=x^{10}-1$, as well as  moment statistics (up to the 10th moment). Observe that the numerical moments $M[a_1]$, which are computed using primes up to $2^{28}$, are quite close to what we obtained for $M[\mu_1]$. See \cite{x10histogram} for an animated histogram of the $a_1$-distribution. The algorithm used to make the histogram is described in \cite{HarveySuth2014} and \cite{HarveySuth2016}. 
\begin{center}
\begin{figure}[h]
    \centering
    	\includegraphics[width=5.25in]{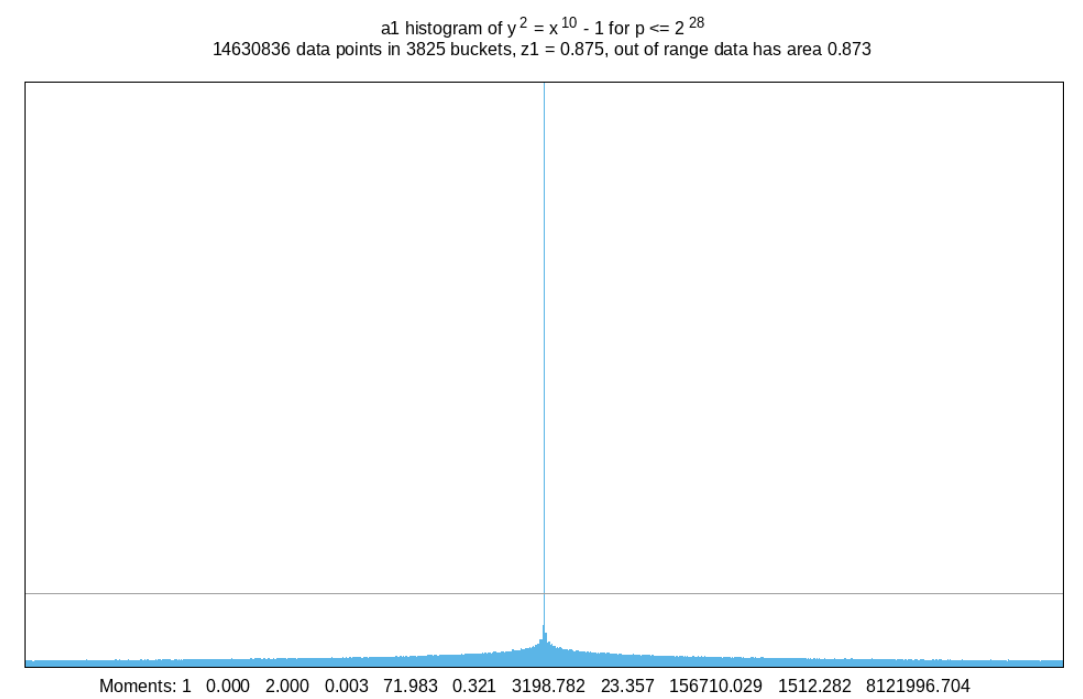}
    \caption{Histogram of $a_1$ values of $y^2=x^{10}-1$ for primes less than $2^{28}$. See \cite{x10histogram}.}
    \label{fig:x10histogram}
\end{figure}
\end{center}
See Table \ref{table:momemntsp} in Section \ref{sec:a1mu1Moments} for  moment statistics  for other curves.

\subsection{Tables of $\mu_1$- and $a_1$-Moment Statistics}\label{sec:a1mu1Moments}

We first consider curves of the form $C_p:\; y^2=x^p-1$. Note that $M[{}^k\mu_1]=0$ for all $0<k<p-1$. 
One can easily determine from Proposition \ref{prop:pcharpolygamma0} that the coefficient of $T$ in $P_{\gamma^0}(T)$ is $\sum_{i=1}^g s_i$, where $s_i=-(u_i+\overline{u_i})$. Hence, 
\begin{align}\label{eqn:mu1momentsp}
    M_n[^0\mu_1]&=\sum_{\alpha_1, \ldots, \alpha_g=0}^n\binom{n}{\alpha_1,\alpha_2,\ldots,\alpha_g} M_{\alpha_1}[s_1]M_{\alpha_2}[s_2]\cdots M_{\alpha_g}[s_g].
\end{align}

Similarly, for curves of the form $C_{2p}:\; y^2=x^{2p}-1$,
\begin{align}\label{eqn:mu1moments2p}
    M_n[^{k,j}\mu_1]&=2^n\sum_{\alpha_1, \ldots, \alpha_{g/2}=0}^n\binom{n}{\alpha_1,\alpha_2,\ldots,\alpha_{g/2}} M_{\alpha_1}[s_1]M_{\alpha_2}[s_2]\cdots M_{\alpha_{g/2}}[s_{g/2}]
\end{align}
whenever $k=j=0$ and $ M_n[^{k,j}\mu_1]=0$ otherwise. 

We used Sage \cite{Sage} to evaluate Equations \eqref{eqn:mu1momentsp} and \eqref{eqn:mu1moments2p}, and then average over the components, to get the  $\mu_1$-moments shown in Table \ref{table:momemntsp}. Note that $M_n[\mu_1]=0$ for all odd $n$, so we omit those values from the table. For comparison, we computed the numerical $a_1$-moments for primes up to $2^{23}$.

\begin{table}[h]
\begin{tabular}{|c|l|l|l|l|l|}
\hline
$m$ & & $M_2$ & $M_4$ & $M_6$ & $M_8$ \\ 
\hline
%5& $\mu_1$ &1 & 9&  100&  1225\\
%& $a_1$ &  $0.998$ &  $  8.969 $ &  $ 99.530 $ &  $  1217.720$\\
%\hline
%7& $\mu_1$ & 1 &  15 &  310 &  7455\\
%& $a_1$ &  $ 0.995 $ &  $ 14.860 $ &  $ 305.198  $ &  $ 7285.552$\\
%\hline
10& $\mu_1$ & 2&72&3200&156800\\
& $a_1$ &  $1.989 $ &  $   71.484 $ &  $  3172.685 $ &  $   155240.208 $\\
\hline
11& $\mu_1$ &  1 &  27 & 1090 &  55195 \\
& $a_1$ &  $0.991 $ &  $  26.425$ &  $ 1049.681  $ &  $  52204.146$\\
\hline
13& $\mu_1$ &  1 & 33 &  1660 &  106785 
 \\
 & $a_1$ &  $ 0.999$ &  $ 33.108 $ &  $  1677.458 $ &  $  108839.689$\\
\hline
14& $\mu_1$ &2&120&9920&954240 \\
& $a_1$ &  $1.982 $ &  $  118.214  $ &  $ 9694.808 $ &  $ 923186.514$\\
\hline
17& $\mu_1$ &  1  &  45 & 3160 & 290605 
 \\
 & $a_1$ & $ 0.991$ &  $ 44.178 $ &  $ 3068.003 $ &  $ 279757.762$\\
\hline
19&  $\mu_1$ &  1 & 51 & 4090 &  432915
\\
& $a_1$ & $ 0.995 $ &  $ 50.601  $ &  $  4040.554 $ &  $ 425599.259$\\
\hline22& $\mu_1$ &2&216&34880&7064960\\
& $a_1$ & $ 1.996 $ &  $  213.572 $ &  $ 34047.140  $ &  $ 6805376.261$\\
\hline
\end{tabular}
\caption{Table of $\mu_1$- and $a_1$-moments for $y^2=x^m-1$ over $\mathbb Q$.}\label{table:momemntsp}
\end{table}

%%%%%%%%%%%%%%%%%%%%%%%%%%%%%%%%%%%%%%%%%%%%%
%%%%%%%%%%%%%%%%%%%%%%%%%%%%%%%%%%%%%%%%%%%%%

\begin{appendices}
\section{Examples of the $\gamma$ Matrix}\label{app:component}

In Table \ref{table:gammas} we give examples of the matrix $\gamma$ from Theorems \ref{theorem:STp} and \ref{thm:ST2p}. These were computed in Sage \cite{Sage} using Sage's chosen generators for   $(\mathbb Z/p\mathbb Z)^{\times}$ and  $(\mathbb Z/2p\mathbb Z)^{\times}$.
{
\begin{table}[h]
\begin{tabular}{|c|c||c|c|}
\hline
$m$ & $\gamma$&$m$ & $\gamma$ \\ 
\hline
%7& $\begin{pmatrix}0&0&I\\J&0&0\\0&I&0 \end{pmatrix}$ 
10& $\begin{pmatrix}0&0&J&0\\0&0&0&I\\I&0&0&0\\0&-J&0&0 \end{pmatrix}$&13& $\begin{pmatrix}0&I&0&0&0&0\\0&0&0&I&0&0\\0&0&0&0&0&I\\0&0&0&0&J&0\\0&0&J&0&0&0\\J&0&0&0&0&0\end{pmatrix}$\\
\hline
11& $\begin{pmatrix}0&I&0&0&0\\0&0&0&I&0\\0&0&0&0&J\\0&0&J&0&0\\J&0&0&0&0\end{pmatrix}$ &14& $\begin{pmatrix}0&0&I&0&0&0\\0&0&0&0&0&I\\0&0&0&0&J&0\\0&-J&0&0&0&0\\I&0&0&0&0&0\\0&0&0&I&0&0 \end{pmatrix}$\\
\hline
%& 17& $\begin{pmatrix}0&0&I&0&0&0&0&0\\0&0&0&0&0&I&0&0\\0&0&0&0&0&0&0&J\\0&0&0&0&J&0&0&0\\0&J&0&0&0&0&0&0\\I&0&0&0&0&0&0&0\\0&0&0&I&0&0&0&0\\0&0&0&0&0&0&I&0\end{pmatrix}$\\
%\hline
\end{tabular}
\caption{Examples of $\gamma$ matrices for $y^2=x^m-1$.}\label{table:gammas}
\end{table}

}

\end{appendices}

%\clearpage
%\FloatBarrier
\bibliographystyle{abbrv}
\bibliography{SatoTatebib2019.bib}

\end{document}